\documentclass[preprint,12pt]{elsarticle}

\usepackage{graphicx}
\usepackage[cp1251]{inputenc}
\usepackage{amssymb}
\usepackage{amsthm}
\usepackage{cmap}

\usepackage{amsmath}

\biboptions{sort&compress}

\newcommand{\sysn}{\left\{\begin{array}{rcl}}
\newcommand{\sysk}{\end{array}\right.}

\newtheorem{theorem}{Theorem}[section]
\newtheorem{lemma}[theorem]{Lemma}

\theoremstyle{example}

\newtheorem{proposition}[theorem]{Proposition}
\theoremstyle{definition}
\newtheorem{definition}[theorem]{Definition}

\newtheorem{corollary}[theorem]{Corollary}

\journal{...}

\begin{document}

\title{On Baire property of spaces of compact-valued
measurable functions}

\author{Alexander V. Osipov}

\address{Krasovskii Institute of Mathematics and Mechanics, \\ Ural Federal
 University, Yekaterinburg, Russia}

\ead{OAB@list.ru}

\begin{abstract} A topological space $X$ is {\it Baire} if the Baire
Category Theorem holds for $X$, i.e., the intersection of any
sequence of open dense subsets of $X$ is dense in $X$. One of the
interesting problems in the theory of functional spaces is the
characterization of the Baire property of a functional space through
the topological property of the support of functions.

In the paper this problem is solved for the space $M(X, K)$
of all measurable compact-valued ($K$-valued) functions defined on a measurable space $(X,\Sigma)$ with the topology of pointwise convergence. It is proved that
$M(X, K)$ is Baire for any metrizable compact space $K$.
\end{abstract}

\begin{keyword}  Baire property \sep function space \sep Baire function

\MSC[2010] 54C35 \sep 54E52 \sep 46A03 \sep 22A05

\end{keyword}

\maketitle 


\section{Introduction}

 A space is {\it meager} (or {\it of the first Baire category}) if it
can be written as a countable union of closed sets with empty
interior. A topological space $X$ is {\it Baire} if the Baire
Category Theorem holds for $X$, i.e., the intersection of any
sequence of open dense subsets of $X$ is dense in $X$. Clearly, if
$X$ is Baire, then $X$ is not meager. The reverse implication is
in general not true. However, it holds for every homogeneous space
$X$ (see Theorem 2.3 in \protect\cite{LM}).

The Baire property is an important topological property used not only in topology but also in other areas of mathematics, such as category theory, functional analysis, measure theory and others, and so it is natural to explore the question: {\it When does a function space have (or not have) the Baire property?}

The Baire property for the space of continuous mappings was first considered in \protect\cite{LM} and \protect\cite{Vid}.
In \protect\cite{LM}, necessary and, in some cases, sufficient conditions on the space $X$ were obtained under which the space $C_p(X,\mathbb{R})$ of all continuous real-valued functions defined on the space $X$ with the topology $p$ of pointwise convergence, has the Baire property.

In general, characterizing the Baire property for function spaces is not a simple task. For example, the problem for the space $C_p(X,\mathbb{R})$ remained unsolved for a long time. Later, the problem for $C_p(X,\mathbb{R})$ was solved independently by E.G.~Pytkeev in \protect\cite{pyt1}, V.~Tkachuk in
\protect\cite{tk} and van Dau in \protect\cite{vD}. For the space
$C_p(X,[0,1])$, the Baire property was studied in \cite{OsPy1}.  For the space $B_1(X,\mathbb{R})$ of Baire-one functions the problem was solved in \cite{Osip1}.

\section{Main definitions and notation}

The set of positive integers is denoted by $\mathbb{N}$ and
$\omega=\mathbb{N}\cup \{0\}$. Let $\mathbb{R}$ be the real line,
we put $\mathbb{I}=[0,1]\subset \mathbb{R}$, and let $\mathbb{Q}$
be the rational numbers.

  Let $V=\{f\in \{0,1\}^X: f(x_i)\in p_i,
i=1,...,n\}$ where $x_i\in X$, $p_i\in \{0,1\}$ for $i=1,...,n$.
Denote by $supp V:=\{x_1,...,x_n\}$.

The basic open neighborhood of a function $f\in\{0,1\}^X $ with support $A$ will be the set $O(f,A):=\{h\in\{0,1\}^X: h\upharpoonright A=f\upharpoonright
A\}$, where $A$ is a finite subset of the space $X$. Note that $supp O(f,A)=A$.

\medskip

Recall that $(X,\Sigma)$ is a {\it  measurable space}, if $X$ is a
set equipped with $\sigma$-algebra $\Sigma$. In mathematical analysis and in probability theory, a {\it $\sigma$-algebra} on a set $X$ is a nonempty collection $\Sigma$ of subsets of $X$ closed under complement, countable unions, and countable intersections.

\medskip
\begin{definition} Let $(X,\Sigma)$ be a measurable space and $A,B\subset X$. Then $A$ and $B$ are {\it $\Sigma$-separated}
if there is $C\in \Sigma$ such that $A\subseteq C$ and $B\subseteq
X\setminus C$.
\end{definition}

\medskip
A function $f: X\rightarrow Y$ is said to be {\it measurable} if
$f^{-1}(V)\in \Sigma$ for every open subset $V$ of $Y$.

A function $f: X\rightarrow \{0,1\}$ is said to be {\it indicator
measurable} if for every $i\in \{0,1\}$ the pre-image
$f^{-1}(i)\in \Sigma$.

\medskip

Denote by ${\bf 2}$ the doubleton, i.e., ${\bf 2}=\{0,1\}$ endowed
with the discrete topology. Denote by $M(X,{\bf 2})$ the set of
all indicator measurable functions with the pointwise convergence
topology. If $Y$ is a topological space then denote by $M(X,Y)$
the set of all measurable functions with the pointwise convergence
topology.

\medskip
Note that the following proposition holds.

\begin{proposition} Let $(X,\Sigma)$ be a measurable space. Then  $M(X,{\bf 2})$ is dense
in ${\bf 2}^X$ if and only if for any points $x,y\in X$ ($x\neq
y$) the sets $\{x\}$ and $\{y\}$ are $\Sigma$-separated.
\end{proposition}

{\it Further, we assume that $\sigma$-algebra $\Sigma$ has the
property that for any points $x,y\in X$ ($x\neq y$) the sets
$\{x\}$ and $\{y\}$ are $\Sigma$-separated.}

\medskip

For other notation almost without exceptions we follow the
Engelking's book \cite{Eng}.

\medskip

\begin{definition} We say that a measurable space $(X,\Sigma)$ has the
$M$-property if every disjunct sequence $\{\Delta_{n} : n\in
\mathbb{N}\}$ of finite subsets $\Delta_n\subset X$,
$\Delta_n=A_n\cup B_n$, $A_n\cap B_n=\emptyset$, $n\in
\mathbb{N}$,
 contains subsequence $\{\Delta_{n_k} : k\in
\mathbb{N}\}$ such that the sets $\bigcup\limits_{k=1}^{\infty}
A_{n_k}$ and $\bigcup\limits_{k=1}^{\infty} B_{n_k}$ are
$\Sigma$-separated.

\end{definition}

The following two lemmas were proved for the space $B(X,{\bf
2})$ of Baire indicator functions in \cite{Osip}. The proofs of
these lemmas for the space $M(X,{\bf 2})$ are repeated by formally
replacing the space $B(X,{\bf 2})$ with the space $M(X,{\bf 2})$.
For the convenience of the reader, we present these proofs.

\begin{lemma}\label{24}
Let $(X,\Sigma)$ be a measurable space. Then the following
statements are equivalent:

\begin{enumerate}

\item  $X$ has the $M$-property;

\item for any disjunct sequence $\{\Delta_{n} : n\in \mathbb{N}\}$
of finite subsets $\Delta_n\subset X$ and a function $f:
\bigcup\limits_{n=1}^{\infty}\Delta_n \rightarrow \{0,1\}$ there
is a subsequence $\{\Delta_{n_k} : k\in \mathbb{N}\}$ and
$\widetilde{f}\in M(X,{\bf 2})$ such that
$\widetilde{f}\upharpoonright \bigcup\limits_{k=1}^{\infty}
\Delta_{n_k}=f\upharpoonright \bigcup\limits_{k=1}^{\infty}
\Delta_{n_k}$.

\end{enumerate}

\end{lemma}

\begin{proof}

$(1)\Rightarrow(2)$. Let $\{\Delta_{n} : n\in \mathbb{N}\}$ be a
disjunct sequence of finite subsets of $X$, $f:
\bigcup\limits_{n=1}^{\infty}\Delta_n \rightarrow \{0,1\}$. Let
$A_n=f^{-1}(0)\cap \Delta_n$, $B_n=f^{-1}(1)\cap \Delta_n$ for
every $n\in \mathbb{N}$. By (1), there is a subsequence
$\{\Delta_{n_k} : k\in \mathbb{N}\}$ such that the sets
$A=\bigcup\limits_{k=1}^{\infty} A_{n_k}$ and
$B=\bigcup\limits_{k=1}^{\infty} B_{n_k}$ are $\Sigma$-separated,
i.e. there is $C\in \Sigma$ such that $A\subseteq C$ and
$B\subseteq X\setminus C$. Consider the function $\widetilde{f}\in
M(X,{\bf 2})$ such that $\widetilde{f}(C)\subseteq \{0\}$ and
$\widetilde{f}(X\setminus C)\subseteq \{1\}$. It is clear that
$\widetilde{f}\upharpoonright \bigcup\limits_{k=1}^{\infty}
\Delta_{n_k}=f\upharpoonright \bigcup\limits_{k=1}^{\infty}
\Delta_{n_k}$.

$(2)\Rightarrow(1)$. Let $\{\Delta_{n} : n\in \mathbb{N}\}$ be a
disjunct sequence of finite subsets of $X$, $\Delta_n=A_n\cup
B_n$, $A_n\cap B_n=\emptyset$, $n\in \mathbb{N}$. Let $f:
\bigcup\limits_{n=1}^{\infty}\Delta_n \rightarrow \{0,1\}$ such
that $f(x)=0$ for $x\in \bigcup\limits_{n=1}^{\infty} A_{n}$ and
$f(x)=1$ for $x\in \bigcup\limits_{n=1}^{\infty} B_{n}$. By (2),
there is a subsequence $\{\Delta_{n_k} : k\in \mathbb{N}\}$ and
$\widetilde{f}\in M(X,{\bf 2})$ such that
$\widetilde{f}\upharpoonright \bigcup\limits_{k=1}^{\infty}
\Delta_{n_k}=f\upharpoonright \bigcup\limits_{k=1}^{\infty}
\Delta_{n_k}$. It is clear that the subsequence $\{\Delta_{n_k} :
k\in \mathbb{N}\}$ required where the sets
$A=\bigcup\limits_{k=1}^{\infty} A_{n_k}$ and
$B=\bigcup\limits_{k=1}^{\infty} B_{n_k}$ are $\Sigma$-separated
by $C=\widetilde{f}^{-1}(0)$.
\end{proof}

\begin{lemma}\label{25}  Let $(X,\Sigma)$ be a measurable space,
$\Delta=\{x_1,...,x_m\}$ be a finite subset of $X$ and $F$ be a
closed nowhere dense subset of $M(X,{\bf 2})$. Then there is an
open base set $V$ in the space $M(X,{\bf 2})$ such that

(a)  $suppV \cap \Delta=\emptyset$;

(b) for any $f\in V$ there is  $\varphi: A \rightarrow \{0,1\}$
where $A:=\Delta\cup supp V$ such that $\varphi=f\upharpoonright
A$ and $\varphi\not\in \overline{\{h\upharpoonright A: h\in
F\}}^{{\bf 2}^{A}}$.

\end{lemma}

\begin{proof}

Renumber all points of the set ${\bf 2}^{\Delta}$ as $\{t_i: 1\leq
i\leq 2^{\Delta}\}$. Denote by $W(\Delta,t_i):=\{f\in M(X,{\bf
2}): f\upharpoonright \Delta=t_i\}$ for every $1\leq i\leq
2^{\Delta}$.

Construct by induction open base sets $W_i$, $V_i$, $1\leq i\leq
2^{\Delta}$ such that

(1) $supp W_i=\Delta$ for every $1\leq i\leq 2^{\Delta}$;

(2) $V_{i+1}\subseteq V_i$, $suppV_{i+1}\supseteq suppV_i$ for
every $1\leq i\leq 2^{\Delta}-1$;

(3) $suppV_i\cap \Delta=\emptyset$;

(4)  $W_{i+1}\cap V_{i+1}\subseteq (W(\Delta, t_{i+1})\cap V_i)\setminus F$.

\medskip

Since $W(\Delta,t_1)$ is open and nonempty, there are open base
sets $W_1$ and $V_1$ such that $W_1\cap V_1\subseteq
W(\Delta,t_1)\setminus F$, $suppW_1=\Delta$, $suppV_1\cap
\Delta=\emptyset$.

Let the sets $W_i$, $V_i$, $1\leq i\leq k<2^{\Delta}$ be
constructed.

Then the set $W(\Delta,t_{k+1})\cap V_k$ is open and non-empty.
Hence, there are open base sets $W_{k+1}$, $V_{k+1}$ such that
$W_{k+1}\cap V_{k+1}\subseteq (W(\Delta,t_{k+1})\cap V_k)\setminus
F$, $suppW_{k+1}=\Delta$, $supp V_{k+1}\supseteq supp V_k$,
$V_{k+1}\subseteq V_k$.

Thus, the sets $W_i$, $V_i$, $1\leq i\leq 2^{\Delta}$ are
constructed.

We show that $V=V_{2^{\Delta}}$ is required. Note that $supp V\cap
\Delta=\emptyset$. Let $f\in V$. Then there is $i_0$ such that
$t_{i_0}=f\upharpoonright \Delta$.

Let $W_{i_0}=\{g\in M(X, {\bf 2}): g(P)\subseteq \{0\}$,
$g(Q)\subseteq \{1\} \}$ for finite subsets $P$ and $Q$ of $X$
such that $P\cap Q=\emptyset$ and $P\cup Q=\Delta$.

Let $A:=\Delta\cup supp V$. Consider the function $\varphi: A
\rightarrow \{0,1\}$ such that $\varphi(P)\subseteq \{0\}$,
$\varphi(Q)\subseteq \{1\}$, $\varphi \upharpoonright suppV=
f\upharpoonright supp V$.

 Since $W_{i_0}\cap V_{i_0}\subseteq W(\Delta, t_{i_0})$ and $V\subseteq V_{i_0}$, $W_{i_0}\cap V\subseteq W(\Delta, t_{i_0})$.
 Note that  $\varphi=f\upharpoonright A$.

Since $(W_{i_0}\cap V)\cap F=\emptyset$, $\varphi\not\in
\overline{\{h\upharpoonright A: h\in F\}}^{{\bf 2}^{A}}$.

\end{proof}

\section{Main results}

\begin{theorem}\label{th1} For any  measurable
space  $(X,\Sigma)$ the space $M(X, {\bf 2})$ is Baire.

\end{theorem}

\begin{proof}

$(1)$ We claim that  $X$ has the $M$-property.  Let $\{\Delta_{n} : n\in \mathbb{N}\}$ be a disjunct sequence of
finite subsets $\Delta_n\subset X$, $\Delta_n=A_n\cup B_n$,
$A_n\cap B_n=\emptyset$, $n\in \mathbb{N}$.

Let $x,y\in \bigcup \limits_{n=1}^{\infty} \Delta_{n}$ and $x\neq
y$. Since the sets $\{x\}$ and $\{y\}$ are $\Sigma$-separated,
there is $C_y\in \Sigma$ such that $x\in C_y$ and $y\in X\setminus
C_y$. Let $C(x)=\bigcap\{C_y: y\in \bigcup \limits_{n=1}^{\infty}
\Delta_{n}\setminus\{x\}\}$. Since $\Sigma$ is a $\sigma$-algebra,
$C(x)\in \Sigma$. Note that $C(x)\cap C(y)=\emptyset$ for $x\neq
y$. Let $C=\bigcup\{C(x): x\in \bigcup\limits_{n=1}^\infty A_n\}$.
Then $C\in \Sigma$. Note that $\bigcup\limits_{n=1}^{\infty}
A_{n}\subseteq C$ and $\bigcup\limits_{n=1}^{\infty}
B_{n}\subseteq X\setminus C$.

$(2)$.  We claim that the space $M(X, {\bf 2})$ is Baire.  Suppose the opposite. Let $M(X, {\bf
2})=\bigcup\limits_{i=1}^{\infty} F_i$, where $F_i$ is a closed
nowhere dense subset of $M(X, {\bf 2})$ and $F_i\subset F_{i+1}$
for every $i\in \mathbb{N}$.

Let $U_1$ be a non-empty open base set such that $U_1\cap
F_1=\emptyset$. Denote by $\Theta_1=supp U_1$. By Lemma 2, for a
finite subset $\Theta_1$ of $X$ and nowhere dense set $F_2$ in
$M(X, {\bf 2})$ there is an open base set $U_2$ such that

(1) $supp U_2\cap \Theta_1=\emptyset$;

(2) for any $f\in U_2$ there is $\varphi\in {\bf 2}^{\Theta_2}$,
where $\Theta_2:=supp U_2\cup \Theta_1$ such that
$\varphi=f\upharpoonright \Theta_2$ and $\varphi\not\in
\overline{\{h\upharpoonright \Theta_2: h\in F_2\}}^{{\bf
2}^{\Theta_2}}$.

For every $k\in \mathbb{N}$, by Lemma \ref{25}, for the finite set
$\Theta_k$ in $X$ and nowhere dense set $F_{k+1}$ in $M(X, {\bf
2})$ there is a base open set $U_{k+1}$ such that

(1) $supp U_{k+1}\cap \Theta_k=\emptyset$;

(2) for any $f\in U_{k+1}$ there is  $\varphi\in {\bf
2}^{\Theta_{k+1}}$ where $\Theta_{k+1}=supp U_{k+1}\cup \Theta_k$
such that $\varphi=f\upharpoonright \Theta_{k+1}$ and
$\varphi\not\in \overline{\{h\upharpoonright \Theta_{k+1}: h\in
F_{k+1}\}}^{{\bf 2}^{\Theta_{k+1}}}$.

Let $\Delta_{1}=supp U_1=\Theta_1$ and $\Delta_{k}=supp
U_k=\Theta_k\setminus \Theta_{k-1}$ for every $k>1$.

For every $k\in \mathbb{N}$ consider $f_k\in U_k$. Define the map
$f: \bigcup\limits_{k=1}^{\infty} \Delta_{k} \rightarrow \{0,1\}$
such that $f\upharpoonright \Delta_{k}=f_k\upharpoonright
\Delta_{k}$ for every $k\in \mathbb{N}$.

Since $X$ has the $M$-property, there are a subsequence $\{n_k: k\in \mathbb{N}\}$ and
$\widetilde{f}\in M(X, {\bf 2})$ such that
$\widetilde{f}\upharpoonright \bigcup\limits_{k=1}^{\infty}
\Delta_{n_k}=f\upharpoonright \bigcup\limits_{k=1}^{\infty}
\Delta_{n_k}$ by Lemma \ref{24}. Then $\widetilde{f}\not\in F_{n_k}$ for every $k\in
\mathbb{N}$. Since $M(X, {\bf 2})=\bigcup\limits_{k=1}^{\infty}
F_{n_k}$, we receive a contradiction.

\end{proof}

\medskip
\begin{corollary}\label{cr1} Let $\{(X_{\alpha}, \Sigma_{\alpha}): \alpha\in
A\}$ be a family measurable spaces. Then $\prod\limits_{\alpha\in A} M(X_{\alpha},{\bf 2})$ is Baire.
\end{corollary}

\begin{proof} Let $\Sigma_{\bigoplus}=\{A\subset \bigoplus\limits_{\alpha\in A} X_{\alpha}: A\cap X_{\alpha}\in \Sigma_{\alpha}$ for each $\alpha\in
A\}$. Then $(\bigoplus\limits_{\alpha\in A} X_{\alpha}, \Sigma_{\bigoplus})$ is a measurable space.

   We claim that $\prod\limits_{\alpha\in A}
M(X_{\alpha},{\bf 2})\cong M(\bigoplus\limits_{\alpha\in A} X_{\alpha},{\bf 2})$.
To this end, we define a mapping $f\rightarrow f^*$ where $f\in \prod\limits_{\alpha\in A}
M(X_{\alpha},{\bf 2})$ and $f^*\in M(\bigoplus\limits_{\alpha\in A} X_{\alpha},{\bf 2})$.

Let $f\in \prod\limits_{\alpha\in A}
M(X_{\alpha},{\bf 2})$. Then $f=(f_{\alpha})$ where $f_{\alpha}\in M(X_{\alpha},{\bf 2})$ for each $\alpha\in A$.

Define $f^*: \bigoplus\limits_{\alpha\in A} X_{\alpha} \rightarrow \{0,1\}$ by letting $f^*\upharpoonright X_{\alpha}=f_{\alpha}$ for each $\alpha\in A$.

We claim that  $f^*\in M(\bigoplus\limits_{\alpha\in A} X_{\alpha},{\bf 2})$. Let $i\in\{0,1\}$. Then $(f^*)^{-1}(i)\cap X_{\alpha}=(f_{\alpha})^{-1}(i)\cap X_{\alpha}\in \Sigma_{\alpha}$ for each $\alpha\in
A$. It follows that $(f^*)^{-1}(i)\in \Sigma_{\bigoplus}$ and, hence, $f^*\in M(\bigoplus\limits_{\alpha\in A} X_{\alpha},{\bf 2})$.

\medskip
The mapping $f\rightarrow f^*$ is one-to-one mapping from  $\prod\limits_{\alpha\in A} M(X_{\alpha},{\bf 2})$ onto $M(\bigoplus\limits_{\alpha\in A} X_{\alpha}, {\bf 2})$.
Furthermore, both the mapping and its inverse are continuous. Thus, the mapping is a homeomorphism.

\medskip
By Theorem \ref{th1}, $M(\bigoplus\limits_{\alpha\in A} X_{\alpha},{\bf 2})$ is Baire, hence, the space $\prod\limits_{\alpha\in A} M(X_{\alpha},{\bf 2})$ is Baire.
\end{proof}

Recall that a mapping $\varphi: K\rightarrow M$ is called {\it
almost open} if $Int \varphi(V)\neq \emptyset$ whenever non-empty
open subset $V$ of $K$. Note that irreducible mappings are almost
open mappings that are defined on compact spaces. Also note that
if $\varphi_{\alpha}:K_{\alpha}\rightarrow M_{\alpha}$ ($\alpha\in
A$) are surjective almost open mappings then the product mapping
$\prod\limits_{\alpha\in A} \varphi_{\alpha}:
\prod\limits_{\alpha\in A} K_{\alpha}\rightarrow
\prod\limits_{\alpha\in A} M_{\alpha}$ is also almost open.

\begin{lemma} (Lemma 3.14 in \cite{OsPy})\label{327} Let $\psi:P\rightarrow L$ be a surjective continuous
almost open mapping and $E\subset P$ be a dense non-meager (Baire)
subspace in $P$. Then $\psi(E)$ is non-meager (Baire).
\end{lemma}

\begin{lemma}(\cite{ArPo})\label{316} Let $L$ be a compact space without isolated
points, $T$ be the Cantor set. Then there exists a surjective
continuous irreducible mapping $\psi:T\rightarrow L$.

\end{lemma}

\begin{lemma}\label{lem6} Let $\psi: K\rightarrow L$ be a continuous almost open surjective mapping. If $M(X,K)$ is a dense subset of $K^X$ and $M(X,K)$ is Baire, then $M(X,L)$ is Baire, too.

\end{lemma}

\begin{proof} The mapping $\psi^X: K^X \rightarrow L^X$ is continuous almost open surjective. Since $M(X,K)$ is a dense subset of $K^X$ and $M(X,K)$ is Baire, $\psi^X(M(X,K))$ is Baire by Lemma \ref{327}. Note that $\psi^X(M(X,K))=\{\psi\circ f: f\in M(X,K)\}\subseteq M(X,L)$. Thus, $M(X,L)$ contains a dense Baire subset $\psi^X(M(X,K))$. It follows that $M(X,L)$ is Baire, too.

\end{proof}

\begin{theorem} Let $(X,\Sigma)$ be a measurable space. Then the space $M(X,L)$ is Baire for any  metrizable compact space $L$.
\end{theorem}

\begin{proof} Note that $M(X,{\bf 2}^{\omega})=(M(X,{\bf
2}))^{\omega}$. Thus, by Corollary \ref{cr1}, $M(X,{\bf
2}^{\omega})$ is Baire.

Let $L$ be a metrizable compact space and let $L_1$ be a
metrizable compact space without isolated points. By Lemma
\ref{316}, there is a surjective continuous irreducible mapping
$\psi: {\bf 2}^{\omega}\rightarrow L\times L_1$. Let $\pi_L: L\times L_1\rightarrow L$ be a projection mapping on $L$.

 Then $\pi_L\circ
\psi:{\bf 2}^{\omega}\rightarrow L$ is a surjective continuous
almost open mapping. By Lemma \ref{lem6}, $M(X,L)$ is Baire.

\end{proof}

\section{Examples}

All examples of function spaces given below we consider with the topology of pointwise convergence. 

\medskip

$\bullet$ An important example is the {\it Borel algebra} over any topological space: the $\sigma$-algebra generated by the open sets (or, equivalently, by the closed sets).

\begin{corollary} {\it Let $X$ be a topological $T_1$-space and $K$ be a metrizable compact space. Then the space $Borel(X,K)$ of all $K$-valued Borel functions defined on $X$ is Baire.}
\end{corollary}

$\bullet$ On the Euclidean space $\mathbb{R}^n$ another $\sigma$-algebra is of importance: that of all Lebesgue measurable sets. This $\sigma$-algebra contains more sets than the Borel $\sigma$-algebra on $\mathbb{R}^n$ and is preferred in integration theory, as it gives a complete measure space.

\begin{corollary}{\it Let $K$ be a metrizable compact space. Then $Lebesgue(\mathbb{R}^n,K)$ of all $K$-valued Lebesgue measurable functions defined on $\mathbb{R}^n$ is Baire.}
\end{corollary}

$\bullet$ An another example is the {\it Baire algebra} over any topological space: the $\sigma$-algebra generated by the co-zero sets (or, equivalently, by the zero-sets).

\medskip

Recall that a topological space is {\it functionally Hausdorff}, if for any two points in it, there is a continuous function from the whole space  to the reals that takes the value $0$ at one point and $1$ at the other.

\begin{corollary}{\it Let $X$ be a functionally Hausdorff space and $K$ be a metrizable compact space. Then the space $Baire(X,K)$ of all $K$-valued Baire functions defined on $X$ is Baire.}
\end{corollary}


\bibliographystyle{model1a-num-names}
\bibliography{<your-bib-database>}

\end{document}